\documentclass[11pt]{amsart}

\usepackage{hyperref}
\usepackage[T1]{fontenc}

\usepackage{tikz}
\usetikzlibrary{arrows}
\usetikzlibrary{decorations.pathreplacing}
\usetikzlibrary{calc}
\usetikzlibrary{fit}
\usetikzlibrary{circuits.logic.US}
\tikzstyle{vert}=[circle,draw=blue!50,fill=blue!20,thick]
\tikzstyle{dot}=[circle,draw,radius=1mm,thick]
\tikzstyle{arw}=[->,shorten <=1pt,>=angle 90,semithick]

\newcommand{\m}[1]{{\mathbf{\uppercase{#1}}}}

\newcommand{\bigO}{O}

\def\compNP{{\textsf{NP}}}
\def\compEXPTIME{{\textsf{EXPTIME}}}
\def\compP{{\textsf{P}}}
\newcommand{\minority}{$\textnormal
  {\textsf{Minority\/}}^{\textnormal{\textrm{Id}}}$}
\newcommand{\smp}{\textnormal{\textsf{SMP}}}
\newcommand{\smpun}{$\smp^{\textnormal{\textrm{Un}}}$}

\theoremstyle{plain}
\newtheorem{theorem}{Theorem}

\newtheorem{observation}[theorem]{Observation}
\newtheorem{corollary}[theorem]{Corollary}
\newtheorem{proposition}[theorem]{Proposition}
\newtheorem{claim}[theorem]{Claim}

\theoremstyle{definition}
\newtheorem{definition}[theorem]{Definition}
\newtheorem*{open-problem*}{Open problem}

\theoremstyle{remark}

\begin{document}

\bibliographystyle{plain}

%

  \author{Alexandr Kazda}
  \address{Charles University, Prague, Czech Republic}
  \email{alex.kazda@gmail.com}

  \author{Jakub Opr\v{s}al}
  \address{University of Durham, Durham, UK}
  \email{jakub.oprsal@durham.ac.uk}

  \author{Matt Valeriote}
  \address{McMaster University, Hamilton, Ontario, Canada}
  \email{matt@math.mcmaster.ca}

  \author{Dmitriy Zhuk}
  \address{Charles University, Prague, Czech Republic
    \and Lomonosov Moscow State University, Moscow, Russia}
  \email{zhuk@intsys.msu.ru}

  \title{Deciding the existence of minority terms}

  \hypersetup{
    pdftitle={Deciding the existence of minority terms},
    pdfauthor={Alexandr Kazda, Jakub Oprsal, Matt Valeriote, and Dmitriy Zhuk},
    pdfborder=0 0 0,
  }


  \begin{abstract}
    This paper investigates the computational complexity of deciding if a~given
    finite idempotent algebra has a~ternary term operation $m$ that satisfies
    the minority equations $m(y,x,x) \approx m(x,y,x) \approx m(x,x,y) \approx
    y$. We show that a~common polynomial-time approach to testing for this type
    of condition will not work in this case and that this decision problem lies
    in the class \compNP.
  \end{abstract}

  \thanks{The first author was supported by the Charles University grants
  PRIMUS/SCI/12 and UNCE/SCI/22. The second author was supported by European Research Council (Grant Agreement no.\ 681988, CSP-Infinity) and UK EPSRC (Grant EP/R034516/1), and the third author was supported by the Natural Sciences and Engineering Council of Canada.}

  \maketitle

%

\section{Introduction}

It is not difficult to see that for the 2-element group $\mathbb{Z}_2 =
\langle\{0,1\}, + \rangle$, the term operation $m(x,y,z) = x + y + z$ satisfies
the equations
\begin{equation}\label{min-eq}
  m(y,x,x) \approx m(x,y,x) \approx m(x,x,y) \approx y.
\end{equation}
A slightly more challenging exercise is to show that a~finite Abelian group
will have such a~term operation if and only if it is isomorphic to a~Cartesian
power of~$\mathbb{Z}_2$.

A ternary operation $m(x,y,z)$ on a~set $A$ is called a~\emph{minority
operation on $A$} if it satisfies the identities~(\ref{min-eq}).  A ternary
term $t(x,y,z)$ of an algebra $\m a$ is a~\emph{minority term of $\m a$} if its
interpretation as an operation on $A$, $t^{\m a}(x,y,z)$, is a~minority
operation on $A$.  Given a~finite algebra $\m a$, one can decide if it has
a~minority term by constructing all of its ternary term operations and checking
to see if any of them satisfy the equations~(\ref{min-eq}).  Since the set of
ternary term operations of $\m a$ can be as big as $|A|^{|A|^3}$, this
procedure will have a~runtime that in the worst case will be exponential in the
size of~$\m a$.

In this paper we consider the computational complexity of testing for the
existence of a~minority term for finite algebras that are  {idempotent}.  An
$n$-ary operation $f$ on a~set $A$ is \emph{idempotent} if it satisfies the
equation $f(x, x, \dots, x) \approx x$ and an algebra is \emph{idempotent} if
all of its basic operations are. We observe that every minority operation is
idempotent. While idempotent algebras are rather special, one can always form
one by taking the \emph{idempotent reduct} of a~given algebra $\m a$.  This is
the algebra with universe $A$ whose basic operations are all of the idempotent
term operations of $\m a$.  It turns out that many important properties of an
algebra and the variety that it generates are governed by its idempotent
reduct~\cite{kearnes-kiss-book}.

The condition of an algebra having a~minority term is an example of a~more
general existential condition on the set of term operations of an algebra
called a~\emph{strong Maltsev condition}.  Such a~condition consists of
a~finite set of operation symbols along with a~finite set of equations
involving them.  An algebra is said to satisfy the condition if for each
$k$-ary operation symbol from the condition, there is a~corresponding $k$-ary
term operation of the algebra so that under this correspondence, the equations
of the condition hold.  For a~more careful and complete presentation of this
notion and related ones, we refer the reader to~\cite{Garcia-Taylor}.

Given a~strong Maltsev condition $\Sigma$, the problem of determining if
a~finite algebra satisfies $\Sigma$ is decidable and lies in the complexity
class \compEXPTIME.  As in the minority term case, one can construct all term
operations of an algebra up to the largest arity of an operation symbol in
$\Sigma$ and then check to see if any of them can be used to witness the
satisfaction of the equations of $\Sigma$.  In general, we cannot do any better
than this, since for some strong Maltsev conditions, it is known that the
corresponding decision problem is {\compEXPTIME}-complete~\cite{freese-valeriote}.

The situation for finite idempotent algebras appears to be better than in the
general case since there are a~number of strong Maltsev conditions for which
there are polynomial-time  procedures to decide if a~finite  idempotent algebra
satisfies them~\cite{freese-valeriote, horowitz-ijac, kazda-valeriote}.  At
present there is no known characterization of these strong Maltsev conditions
and we hope that the results of this paper may help to lead to a~better
understanding of them.  We refer the reader to~\cite{Bu-Sa} or
to~\cite{bergman-book} for background on the basic algebraic notions  and
results used in this work.

\section{Formulation of the problem}

In this section, we formally introduce the considered problem. In all the
problems mentioned in the introduction, we assume that the input algebra is
given as a list of tables of its basic operations. In particular, this implies
that the input algebra has finitely many operations. We also assume that the
input algebra has at least one operation (i.e., the input is non-empty) and we
forbid nullary operations on the input.
The main concern of this paper is the following decision problem.

\begin{definition}
Define \minority\ to be the following decision problem:
\begin{itemize}
  \item INPUT: A~list of tables of basic operations of an idempotent algebra~$\m A$.
  \item QUESTION: Does $\m a$ have a~minority term?
\end{itemize}
\end{definition}

The size of an input is measured by the following formula. For a finite
algebra~$\m a$, let
\[
  \|\m a\| = \sum_{i = 1}^\infty k_i|A|^i,
\]
where $k_i$ is the number of $i$-ary basic operations of $\m a$. Since we
assume that $\m a$ has only finitely many operations, the sum is finite. Also
note that $\lVert \m a\rVert \geq \lvert A \rvert$ since we assumed that $\m a$
has a non-nullary operation.

\section{Minority is a join of two weaker conditions}
  \label{join}

One approach to understanding the minority term condition is to see if maybe
there exist two weaker Maltsev conditions $\Sigma_1$ and $\Sigma_2$ such that a
finite algebra $\m A$ has a minority term if and only if $\m a$ satisfies both
$\Sigma_1$ and $\Sigma_2$. In this situation, we would say that the minority
term condition is the join of $\Sigma_1$ and $\Sigma_2$. Were this the case, we
could decide if $\m a$ has a minority term by deciding $\Sigma_1$ and
$\Sigma_2$.

On the surface, the minority term condition is already quite concise and
natural; it is not clear if having a minority term can be expressed as a join
of weaker conditions. In this section, we show that it is a join of having a
Maltsev term with a condition which we call having a minority-majority term
(not to be confused with the `generalized minority-majority' terms
from~\cite{GMM-paper}). Maltsev terms are a classical object of study in
universal algebra -- deciding if an
algebra has them is in \compP{} for finite idempotent algebras. The
minority-majority terms are much less understood.

\begin{definition}
  A ternary term $p(x,y,z)$ of an algebra $\m a$ is a~\emph{Maltsev term for
  $\m a$} if it satisfies the equations
  \[
    p(x,x,y)\approx p(y,x,x)\approx y
  \]
  and a~6-ary term $t(x_1, \dots, x_6)$ is a~\emph{minority-majority term} of
  $\m a$ if it satisfies the equations
  \begin{align*}
    t(y,x,x,z,y,y)&\approx y\\
    t(x,y,x,y,z,y)&\approx y\\
    t(x,x,y,y,y,z)&\approx y.
  \end{align*}
\end{definition}

We point out that if an algebra has a~minority term then it also, trivially,
has a~Maltsev term, but that the converse does not hold (as witnessed by the
cyclic group $\mathbb{Z}_4$).  Our definition of a~minority-majority term is
a~strengthening of the term condition found by Ol\v{s}\'{a}k
in~\cite{Olsak2017}. Ol\v{s}\'{a}k has shown that his terms are a~weakest
non-trivial strong Maltsev condition whose terms are all idempotent.

We observe that by padding variables, any algebra that has a~minority term or
a~majority term (just replace the final occurrence of the variable $y$ in the
equations~(\ref{min-eq}) by the variable $x$ to define such a~term) also has
a~minority-majority term.  Since the 2-element lattice has a~majority term but
no minority term, it follows that having a~minority-majority term is strictly
weaker than having a~minority term.

\begin{theorem}\label{thm:join}
  An algebra has a~minority term if and only if it has a~Maltsev term and
  a~minority-majority term.
\end{theorem}

\begin{proof}
The discussion preceding this theorem establishes one direction of this
theorem.  For the other we need to show that if an algebra $\m a$ has a~Maltsev
term  $p(x,y,z)$,  and a~minority-majority term $t(x_1, \dots, x_6)$ then $\m
a$ has a~minority term.  Given such an algebra $\m a$, define
\[
  m(x,y,z)=t(x,y,z,p(z,x,y),p(x,y,z),p(y,z,x)).
\]
Verifying that $m(x,y,z)$ is a~minority term for $\m a$  is
straightforward; we show one of the three required equalities here as an example:
\begin{align*}
  m(x,x,y)&\approx t(x,x,y,p(y,x,x),p(x,x,y),p(x,y,x))\\
  &\approx t(x,x,y,y,y,p(x,y,x))\approx y.
\end{align*}
\end{proof}

\begin{corollary}
  The problem of deciding if a~finite algebra has a~minority term can be
  reduced to the problems of deciding if it has a~Maltsev term and if it has
  a~minority-majority term.
\end{corollary}

As was demonstrated in~\cite{freese-valeriote, horowitz-ijac}, there is
a~polynomial-time algorithm to decide if a~finite idempotent algebra has
a~Maltsev term. Therefore, should testing for a~minority-majority term for
finite idempotent algebras prove to be tractable, then this would lead to
a~fast algorithm for testing for a~minority term, at least for finite
idempotent algebras.  From the hardness results found
in~\cite{freese-valeriote} it follows that in general, the problem of deciding
if a~finite algebra has a~minority-majority term is \compEXPTIME-complete; the
complexity of this problem restricted to idempotent algebras is unknown.

\section{Local Maltsev terms}
  \label{maltsev}

In~\cite{freese-valeriote, horowitz-ijac,  kazda-valeriote,
Valeriote_Willard_2014} polynomial-time algorithms are presented for deciding
if certain Maltsev conditions hold in the variety generated by a~given finite
idempotent algebra.  One particular Maltsev condition that is addressed by all
of these papers is that of having a~Maltsev term.  In all but
\cite{freese-valeriote}, the polynomial-time algorithm produced is based on
testing for the presence of enough `local' Maltsev terms in the given
algebra.

\begin{definition}
  Let $\m a$ be an algebra and $S \subseteq A^2\times\{0,1\}$.  A term
  operation $t(x,y,z)$ of $\m a$ is a~\emph{local Maltsev term operation for
  $S$} if:
  \begin{itemize}
    \item whenever $((a,b), 0) \in S$, $t(a,b,b) = a$, and
    \item whenever $((a,b), 1) \in S$, $t(a,a,b) = b$.
  \end{itemize}
\end{definition}

Clearly, if $\m a$ has a~Maltsev term then it has a~local Maltsev term
operation for every subset $S$ of $A^2 \times\{0,1\}$ and conversely, if $\m a$
has a~local Maltsev term operation for $S = A^2 \times\{0,1\}$ then it has
a~Maltsev term. In~\cite{horowitz-ijac,  kazda-valeriote,
Valeriote_Willard_2014} it is shown that if a~finite idempotent algebra $\m a$
has local Maltsev term operations for all two element subsets of $A^2
\times\{0,1\}$ then $\m a$ will have a~Maltsev term. This fact is then used as
the basis for a~polynomial-time test to decide if a~given finite idempotent
algebra has a~Maltsev term.

In this section we extract an additional piece of information from this
approach to testing for a~Maltsev term, namely that if a~finite idempotent
algebra has a~Maltsev term, then we can produce an operation table or a~circuit
for a~Maltsev term operation in time polynomial in the size of the algebra.
 We will first prove that there is an
algorithm for producing circuits for a Maltsev function; the algorithm for
producing the operation table will then be given as a~corollary. However, for
the reduction presented in Section~\ref{sec:np} we need only the algorithm
for producing a function table.

%

Let us first briefly describe how to get a~global Maltsev operation from local
ones. Assume we know (circuits of) a~local Maltsev term operation
$t_{a,b,c,d}(x,y,z)$ for each two element subset
\[
  \{((a,b), 0), ((c,d), 1)\}
\]
of $A^2 \times\{0,1\}$. These are required for $\m A$ to have a~Maltsev term.
A global Maltsev term can be constructed from them in two stages: First, we construct, for each $a,b\in
A$, an operation $t_{a,b}$ such that $t_{a,b}(a,b,b) = a$ and $t_{a,b}(x,x,y) = y$ for all
$x,y \in A$. This is done by fixing an enumeration $(a_1, b_1)$, $(a_2,
b_2)$, \dots, $(a_{n^2}, b_{n^2})$ of $A^2$, and then defining, for $1 \le j
\le n^2$, the operation $t_{a,b}^j(x,y,z)$ on $A$ inductively as follows:
\begin{itemize}
  \item $t_{a,b}^1(x,y,z) = t_{a,b,a_1, b_1}(x,y,z)$, and
  \item for $1 \le j < n^2$, $t_{a,b}^{j+1}(x,y,z) =
    t_{a,b,u,v}(t_{a,b}^j(x,y,z), t_{a,b}^j(y,y,z), z)$, where $u =
    t_{a,b}^j(a_{j+1}, a_{j+1}, b_{j+1})$ and $v = b_{j+1}$.
\end{itemize}
An easy inductive argument shows that  $t_{a,b}^j(a,b,b) = a$ and
$t_{a,b}^j(a_i, a_i, b_i) = b_i$ for all $i \le j \le n^2$, and so setting
$t_{a,b}(x,y,z) = t_{a,b}^{n^2}(x,y,z)$ works.

In the second stage, we construct a~term $t_j(x,y,z)$ such that $t_j(a,a,b) =
b$ for all $a$, $b \in A$ and $t_j(a_i, b_i, b_i) = a_i$ for all $i \le j$.  We
define this sequence of operations inductively again:
\begin{itemize}
  \item $t_1(x,y,z) = t_{a_1, b_1}(x,y,z)$, and
  \item for $1 \le j < n^2$, $t_{j+1}(x,y,z) = t_{u,v}(x, t_j(x,y,y),
    t_j(x,y,z))$, where $u = a_{j+1}$ and $v = t_j(a_{j+1}, b_{j+1},
    b_{j+1})$.
\end{itemize}
Again, it can be shown that for $1 \le j \le n^2$, the operation $t_j(x,y,z)$
satisfies the claimed properties and so $t_{n^2}(x,y,z)$ will be a~Maltsev term
operation for $\m a$.

From the above construction, one can obtain a~term that
 represents a Maltsev term operation of the algebra $\m A$, starting with terms representing the operations $t_{a,b,c,d}$. But there is
an efficiency problem with this approach:
the term is extended by one layer
in each step, which results in a term of exponential size. Therefore, the
bookkeeping of this term would increase the running time of the algorithm
beyond polynomial. Nevertheless, this can be circumvented by constructing
a~succint representation of the term operations, namely by considering circuits
instead of terms.

Informally, a~circuit over an algebraic language (as a~generalization of
logical circuits) is a~collection of gates labeled by operation symbols, where
the number of inputs of each gate corresponds to the arity of the operation
symbol. The inputs are either connected to outputs of some other gate, or
designated as inputs of the circuit; an output of one of the gates is
designated as an~output of the circuit. Furthermore, these connections allow
for straightforward evaluation, i.e., there are no oriented cycles.

Formally, we define an $n$-ary \emph{circuit} in the language of an algebra $\m A$ as a~directed acyclic graph with possibly multiple edges that has two kinds of vertices: \emph{inputs} and \emph{gates}. There are exactly $n$ inputs, labeled by variables $x_1,\dots, x_n$, and each of them is a~source, and a~finite number of gates. Each gate is labeled by an~operation symbol of $\m A$, the in-degree corresponds to the arity of the operation, and the in-edges are ordered. One of the vertices is designated as the \emph{output} of the circuit. We define the size of the circuit to be the number of its vertices.

The value of a~circuit given an input tuple $a_1,\dots,a_n$ is defined by the following
recursive computation: The value on an input vertex labeled by $x_i$ is $a_i$,
the value on a~gate labeled by $g$ is the value of the operation $g^{\m A}$
applied to the values of its in-neighbours in the specified order. Finally, the
output value of the circuit is the value of the output vertex. It is easy
to see that the value of a~circuit on a~given tuple can be computed in linear
time (in the size of the circuit) in a~straightforward way. For a~fixed circuit
the function that maps the input tuple to the output is a~term function of $\m
A$. Indeed, to find such a~term it is enough to evaluate the circuit in the
free (term) algebra on the tuple $x_1,\dots,x_n$. The converse is also true
since any term can be represented as a~`tree' circuit (it is an oriented tree
if we omit all input vertices). Many terms can be expressed by considerably
smaller circuits. We give one such example in
Figure~\ref{fig:term-and-circuit}.

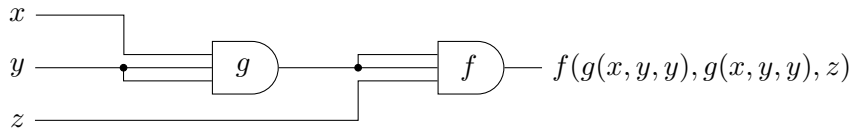
\begin{figure}
  \[ \begin{tikzpicture}[ split/.style = { shape = circle, draw, fill, inner
      sep = 0, minimum size = 2.5 }, circuit logic US ]

    \node [and gate, inputs = {nnn}, point right, minimum size = .7cm] (f) at (3,0) {$f$};
    \node [and gate, inputs = {nnn}, point right, minimum size = .7cm] (g) at (0,0) {$g$};

    \node (x) at (-3,.7) {$x$};
    \node (y) at (-3,0) {$y$};
    \node (z) at (-3,-.7) {$z$};

    \draw (f.output) -- ++(right:.5) node [right] {$f(g(x,y,y),g(x,y,y),z)$} ;

    \draw (g.output) -- (f.input 2);
    \draw ($(g.output)!.5!(f.input 2)$) |- (f.input 1);
    \node [split] at ($(g.output)!.5!(f.input 2)$) {};

    \draw (x.east) -| ($(x.east)!.5!(g.input 1)$) |- (g.input 1);

    \draw (y.east) -- (g.input 2);
    \draw ($(y.east)!.5!(g.input 2)$) |- (g.input 3);
    \node [split] at ($(y.east)!.5!(g.input 2)$) {};

    \draw (z.east) -| ($(g.output)!.5!(f.input 2) + (-90:.5) $) |- (f.input 3);

  \end{tikzpicture} \]
  \caption{A~succinct circuit representation of the term $f(g(x,y,y),g(x,y,y),z)$.}
    \label{fig:term-and-circuit}
\end{figure}

In the proof of the theorem below, we will also use circuits with multiple outputs. The only difference in the definition is that several vertices are designated as outputs. Any such circuit then computes a~tuple of term functions.

\begin{theorem}\label{maltsevcircuit}
Let $\m a$ be a~finite idempotent algebra.  There is an algorithm whose
runtime can be bounded by a~polynomial in the size of $\m a$ that will either
(correctly) output that $\m a$ has no Maltsev term operation, or output
a~circuit for some Maltsev term operation of $\m a$.
\end{theorem}

\begin{proof}
 Let $n =
  |A|$.  Recall that $\m A$ has at least one
  basic operation of positive arity and hence $\|\m A\|\geq n$. Let $m\geq 1$ be the
  maximal arity of an operation of $\m A$.

  We construct a~circuit representing a~Maltsev operation in three steps: The
  first step produces, for each $a$, $b$, $c$, $d$ from $A$, a circuit that computes a local Maltsev term operation $t_{a,b,c,d}$ as defined near the beginning of this section, the second step
  produces circuits that compute $t_{a,b}$, and the final step produces
  a~circuit for a~Maltsev operation $t$. We note that the algorithm can fail
  only in the first step.

  \begin{figure}
  \[
  \begin{tikzpicture}[ split/.style = { shape = circle, draw, fill, inner sep = 0,
      minimum size = 2.5 }, node distance = 1.7cm, circuit logic US  ]
    \node [and gate, inputs = {nnn}, point right] (out1) at (0,1) {$t_{a,b,u,v}$};
    \node [and gate, inputs = {nnn}, point right] (out2) at (0,-1) {$t_{a,b,u,v}$};

    \draw (out1.output) -- ++(right:.5) node [right] {$t^{j+1}_{a,b}(x,y,z)$} ;
    \draw (out2.output) -- ++(right:.5) node [right] {$t^{j+1}_{a,b}(y,y,z)$} ;

    \node at ($(out1.input 1) + (left:2.5)$) (j1) {$t^j_{a,b}(x,y,z)$};
    \node at ($(out2.input 1) + (left:2.5)$) (j2) {$t^j_{a,b}(y,y,z)$};

    \draw (j1.east) -- (out1.input 1);

    \draw (j2.east) -- (out2.input 1);
    \node (split1) [ split ] at ($(j2.east)!.5!(out2.input 1)$) {};
    \draw (split1) |- (out1.input 2);
    \draw (split1) |- (out2.input 2);

    \node [split] (split2) at ($(split1) + (-.25,0) - (out1.input 1) + (out1.input 3)$) {};
    \draw (split2) |- (out1.input 3);
    \draw (split2) -- (out2.input 3);

    \node [ left of = j1 ] (out1j){};
    \draw (j1.west) -- (out1j.east);
    \node [ left of = j2 ] (out2j){};
    \draw (j2.west) -- (out2j.east);

    \node (x) at (-8,1.5) {$x$};
    \node (y) at (-8,0) {$y$};
    \node (z) at (-8,-1.5) {$z$};

    \node [ right of = x ] (in1j){};
    \node [ right of = y ] (in2j){};
    \node [ right of = z ] (in3j){};

    \draw (x.east) -- (in1j);
    \draw (y.east) -- (in2j);
    \draw (z.east) -- (in3j);

    \node (split3) [split] at ($(z.east) !.5! (in3j.west)$) {};
    \node (mid) [ below of = j2 ] {$z$};
    \draw (split3) |- (mid) -| (split2);

    \node [ draw, dashed, fit = (in1j) (in2j) (in3j) (out1j) (out2j) ] {$C^j_{a,b}$};
  \end{tikzpicture}
  \]
    \caption{Recursive definition of circuit $C^{j+1}_{a,b}$.}
    \label{fig:C-j-ab}
  \end{figure}
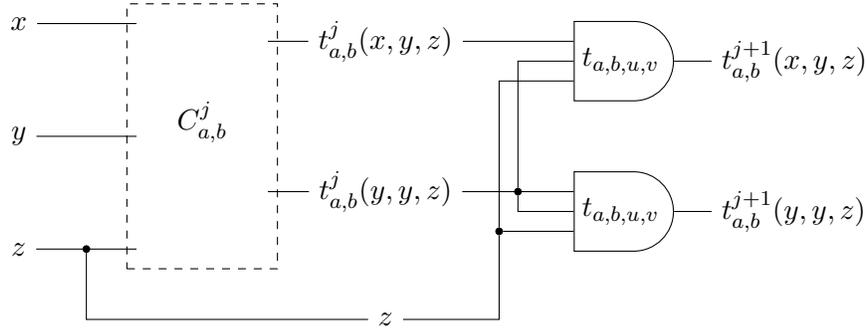

  Step 1: Circuits for $t_{a,b,c,d}$. For each $a,b,c,d$, we aim to produce a~circuit that computes a local Maltsev term
  operation $t_{a,b,c,d}$. To do this, we consider the subuniverse $R$ of $\m
  A^2$ generated by $\{(a,c), (b,c), (b,d)\}$.
  According to Proposition~6.1 from~\cite{freese-valeriote} $R$ can be generated in time $O(||\m a||^2m)$.
  It is clear that $\m A$ has
  a~local Maltsev term operation $t_{a,b,c,d}$ if and only if $(a,d) \in R$. Our algorithm produces a circuit for $t_{a,b,c,d}$ by generating elements of $R$ one at a time and keeping track of circuits that witness
  the membership of these elements.

  More precisely, we employ a subuniverse generating algorithm to produce a sequence
  $r_1 = (a,c), r_2 = (b,c), r_3 = (b,d), r_4, \dots$ of elements of $R$ (in time $O(||\m a||^2m)$) such
  that each $r_{k+1}$, for $k \ge 3$, is obtained from $r_1,\dots,r_{k}$ by a~single
  application of an operation $f$ of $\m A^2$.  Our algorithm will also produce  a~sequence of
  ternary circuits $C_{a,b,c,d}^3 \subseteq C_{a,b,c,d}^4 \subseteq \dots$ such
  that each $C_{a,b,c,d}^k$ has $k$ outputs, and the values of $C_{a,b,c,d}^k$
  on $r_1,r_2,r_3$ give $r_1,\dots,r_k$. We define $C_{a,b,c,d}^3$ to be
  the~circuit with no gates, and outputs $x_1$, $x_2$, $x_3$. The circuit
  $C_{a,b,c,d}^{k+1}$ is defined inductively from $C_{a,b,c,d}^k$: Consider an
  operation $f$ and $r_{i_1},\dots,r_{i_p}$ with $i_j \leq k$ such that
  $r_{k+1} = f(r_{i_1},\dots,r_{i_p})$; add a~gate labeled $f$ to
  $C_{a,b,c,d}^k$ connecting its inputs with the outputs of $C_{a,b,c,d}^k$
  numbered by $i_j$ for $j = 1,\dots, p$. We designate the output of this gate
  as the $(k+1)$-st output of $C_{a,b,c,d}^{k+1}$.

  It is straightforward to
  check that the circuits $C_{a,b,c,d}^k$ satisfy the requirements. We also
  note that the size of $C_{a,b,c,d}^k$ is exactly $k$. We stop this inductive
  construction at some step $k$ if $r_k = (a,d)$, in which case we produce the circuit
  $C_{a,b,c,d}$  from $C_{a,b,c,d}^k$ by indicating a~single output to be the
  $k$-th output of $C_{a,b,c,d}^k$.  If, on the other hand, we have generated all of $R$ without producing $(a,d)$ at any step then the algorithm halts and outputs that $\m a$ does not have a Maltsev term operation.  The soundness of our algorithm follows from the fact that $\m a$ has a~local Maltsev term $t_{a,b,c,d}$ if and only if $(a,d) \in R$
  and that $\m a$ has a Maltsev term if and only if it has local Maltsev terms $t_{a,b,c,d}$ for all $a$, $b$, $c$, $d \in A$.
   The algorithm produces circuits of
  size $\bigO(n^2)$ and spends most of its
  time generating new elements of $R$;
  generating each $C_{a,b,c,d}$ takes time
  $O(\|\m a\|^2m)$, making the total time complexity of Step 1 to be $\bigO(\|\m
  a\|^2mn^4)$.

  Step 2: Circuits for $t_{a,b}$. At this point we assume that the functions $t_{a,b,c,d}$ are part of
  the signature. It is clear that the full circuit can be obtained by
  substituting the circuits $C_{a,b,c,d}$ for gates labeled by $t_{a,b,c,d}$,
  and this can be still done in polynomial time.

  Our task is to obtain a~circuit for $t_{a,b}$. We do this by inductively constructing circuits $C^j_{a,b}$ that compute two
  values of the terms $t_{a,b}^j$, namely $t_{a,b}^j(x,y,z)$ and
  $t_{a,b}^j(y,y,z)$. Starting with $j = 0$ and $t^0(x,y,z) = x$, we define
  $C_{a,b}^0$ to be the circuit with no gates and outputs $x,y$. Further, we
  define circuit $C_{a,b}^{j+1}$ inductively from $C_{a,b}^j$ by adding two
  gates labeled by $t_{a,b,u,v}$, where $u = t_{a,b}^j(a_{j+1}, a_{j+1},
  b_{j+1})$ and $v = b_{j+1}$: the first gate has as inputs the two outputs of
  $C_{a,b}^j$ and $z$, the second gate has as inputs two copies of the second
  output of $C_{a,b}^j$ and $z$. See Figure~\ref{fig:C-j-ab} for a~graphical
  representation. Again, it is straightforward to check that these circuits
  have the required properties.  Also note that the size of $C^j_{a,b}$ is
  bounded by $2j+3$ which is a~polynomial. The final circuit $C_{a,b}$
  computing $t_{a,b}$ is obtained from $C_{a,b}^{n^2}$ by designating the first
  output of $C_{a,b}^{n^2}$ to be the only output of $C_{a,b}$. Once we have
  $t_{a,b,c,d}$ in the signature, this process will run in time $\bigO(n^2)$.

  Step 3: Circuit for a Maltsev term. Again, we assume that $t_{a,b}$ are basic operations, and construct
  circuits $C_j$ computing two values $t_j(x,y,y)$ and $t_j(x,y,z)$ of $t_j$
  inductively. The proof is analogous to Step 2, with the only difference that
  we use Figure~\ref{fig:C-j} for the inductive definition. Again the time
  complexity is $\bigO(n^2)$.

  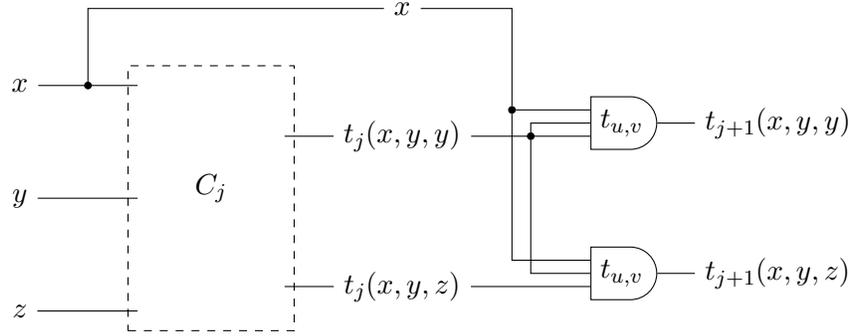
\begin{figure}
    \begin{centering}
    \begin{tikzpicture}[ split/.style = { shape = circle, draw, fill, inner sep = 0,
        minimum size = 2.5 }, node distance = 1.7cm, circuit logic US  ]
      \node [and gate, inputs = {nnn}, point right, minimum size = .7cm] (out1) at (0,1) {$t_{u,v}$};
      \node [and gate, inputs = {nnn}, point right, minimum size = .7cm] (out2) at (0,-1) {$t_{u,v}$};

      \draw (out1.output) -- ++(right:.5) node [right] {$t_{j+1}(x,y,y)$} ;
      \draw (out2.output) -- ++(right:.5) node [right] {$t_{j+1}(x,y,z)$} ;

      \node at ($(out1.input 3) + (left:2.5)$) (j1) {$t_j(x,y,y)$};
      \node at ($(out2.input 3) + (left:2.5)$) (j2) {$t_j(x,y,z)$};
      \node (mid) [ above of = j1 ] {$x$};

      \draw (j1.east) -- (out1.input 3);
      \draw (j2.east) -- (out2.input 3);

      \node (split1) [ split ] at ($(j1.east)!.5!(out1.input 3)$) {};
      \draw (split1) |- (out1.input 2);
      \draw (split1) |- (out2.input 2);

      \node [split] (split2) at ($(split1) + (-.25,0) + (out1.input 1) - (out1.input 3)$) {};
      \draw (split2) -- (out1.input 1);
      \draw (split2) |- (out2.input 1);

      \node [ left of = j1 ] (out1j){};
      \draw (j1.west) -- (out1j.east);
      \node [ left of = j2 ] (out2j){};
      \draw (j2.west) -- (out2j.east);

      \node (x) at (-8,1.5) {$x$};
      \node (y) at (-8,0) {$y$};
      \node (z) at (-8,-1.5) {$z$};

      \node [ right of = x ] (in1j){};
      \node [ right of = y ] (in2j){};
      \node [ right of = z ] (in3j){};

      \draw (x.east) -- (in1j);
      \draw (y.east) -- (in2j);
      \draw (z.east) -- (in3j);

      \node (split3) [split] at ($(x.east) !.5! (in1j.west)$) {};
      \draw (split3) |- (mid) -| (split2);

      \node [ draw, dashed, fit = (in1j) (in2j) (in3j) (out1j) (out2j) ] {$C_j$};
    \end{tikzpicture}
    \end{centering}
    \caption{Recursive definition of circuit $C_{j+1}$.}
    \label{fig:C-j}
  \end{figure}

  Each step runs in time polynomial in $\|\m a\|$ (the time complexity is
  dominated by Step 1) and outputs a~polynomial
  size circuit. This also implies that expanding the gates according to their
  definitions in Steps 2 and 3 can be done in polynomial time; the final size of
  the output circuit will be bounded by $\bigO(n^6)$.
\end{proof}

\begin{corollary}\label{maltsevterm}
  Let $\m a$ be a~finite idempotent algebra.  There is an algorithm whose
  runtime can be bounded by a~polynomial in the size of $\m a$ that will
  produce the table of some Maltsev term operation of $\m a$, should one exist.
\end{corollary}

\begin{proof}
  The polynomial-time algorithm is as follows. First, generate a~polynomial size
  circuit for some Maltsev term operation of $\m a$. This can be done in polynomial time by
  the above theorem. Second, evaluate this circuit at all $\lvert A\rvert^3$
  possible inputs. The second step runs in polynomial time since evaluation of
  a~circuit is linear in the size of the circuit.
\end{proof}

We note that there is also a more straightforward algorithm for producing the
operation table of a Maltsev term which follows the circuit construction but
instead of circuits, it remembers the tables for each of the relevant term
operations.

\section{Local minority terms}

In contrast to the situation for Maltsev terms highlighted in the previous
section, we will show that having plenty of `local' minority terms does not
guarantee that a~finite idempotent algebra will have a~minority term.  One
consequence of this is that an approach along the lines in~\cite{horowitz-ijac,
kazda-valeriote, Valeriote_Willard_2014} to finding an efficient algorithm to
decide if a~finite idempotent algebra has a~minority term will not work.

In this section, we will construct for each odd natural number $n > 2$ a~finite
idempotent algebra $\m a_n$ with the following properties: The universe of $\m
a_n$ has size $4n$ and $\m a_n$ does not have a~minority term, but for every
subset $E$ of $A_n$ of size $n-1$ there is a~term of $\m a_n$ that acts as
a~minority term on the elements of $E$.

We start our construction by fixing some odd $n > 2$ and some minority
operation $m$ on the set $[n] = \{1, 2, \dots, n\}$. To make things concrete
we set
\[
  m(x,y,z)=\begin{cases}
    x& y=z\\
    y& x=z\\
    z& \text{else,}
  \end{cases}
\]
but note that any minority operation on $[n]$ will do.

Since there are two nonisomorphic groups of order 4, we have two different
natural group operations on $\{0,1,2,3\}$: addition modulo~4, which we will
denote by `$+$' (its inverse is `$-$'), and bitwise XOR, which we denote by
`$\oplus$' (this operation takes bitwise XOR of the binary representations of
input numbers, so for example $1\oplus 3=2$). Throughout this section, we will
use arithmetic modulo 4, e.g., $6x = x + x$, for all expressions except those
involving indices.

The construction relies on similarities and subtle differences of the two group
structures, and the derived Maltsev operations, $x-y+z$ and $x\oplus y\oplus z$.
Both these operations share a congruence $\equiv_2$ that is given by taking
the remainder modulo 2. We note that $x\equiv_2 y$ if and only if $2x = 2y$.

\begin{observation}\label{obs:maltsev-diff}
  Let $x,y,z\in \{0,1,2,3\}$. Then 
\[
  (x\oplus y\oplus z) - (x - y + z) \in \{0,2\}
,\]
and moreover the result depends only on the classes of
  $x$, $y$, and $z$ in the congruence $\equiv_2$ (i.e., the least significant
binary bits of $x$, $y$, and $z$).
\end{observation}
\begin{proof}
Both Maltsev operations agree modulo $\equiv_2$, hence the difference lies in
the $\equiv_2$-class of 0.

To see the second part, it is enough to observe that $x\oplus 2=x+2=x-2$ for all
$x$. Hence changing, say $x$ to $x'=x\oplus 2$ simply flips the most
  significant binary bit
of both  $x\oplus y\oplus z$ and $x - y + z$, keeping the difference the same.
\end{proof}


\begin{definition}
Let $A_n=[n]\times [4]$. For $i \in [n]$, we define  $t_i(x,y,z)$ to be the
following operation on $A_n$:
\[
  t_i((a_1,b_1), (a_2,b_2), (a_3,b_3)) =
  \begin{cases}
    (i,b_1 - b_2 + b_3)
      \quad\text{if $a_1 = a_2 = a_3 = i$, and} \\
    (m(a_1,a_2,a_3),b_1\oplus b_2 \oplus b_3),
      \quad\text{otherwise.}
  \end{cases}
\]
The algebra $\m a_n$ is defined to be the algebra with universe $A_n$ and
basic operations $t_1,\dots,t_n$.
\end{definition}

By construction, the following is true.

\begin{claim}\label{local}
  For every $(n-1)$-element subset $E$ of $A_n$, there is a~term operation of
  $\m a_n$ that satisfies the minority term equations when restricted to
  elements from $E$.
\end{claim}

\begin{proof}
  Pick $i\in [n]$ such that no element of $E$ has its first coordinate equal to
  $i$; the operation $t_i$ is a local minority for this $E$.
\end{proof}

\begin{proposition}\label{prop:An}
  For $n > 1$ and odd, the algebra $\m a_n$ does not have a~minority term.
\end{proposition}

\begin{proof}
  Given some $(i,a)\in A_n$, we will refer to $a$ as the \emph{arithmetic part}
  of $(i,a)$. This is to avoid talking about `second coordinates' in
  the confusing situation when $(i,a)$ itself is a~part of a~tuple of elements
  of $A_n$.

  To prove the proposition, we will define a~certain subuniverse $R$ of $(\m
  a_n)^{3n}$ and then show that $R$ is not closed under any minority operation
  on $A_n$ (applied coordinate-wise). We will write $3n$-tuples of elements of
  $A_n$ as $3n\times 2$ matrices where the arithmetic parts of the elements
  make up the second column.

  Let $R \subseteq (A_n)^{3n}$ be the set of all $3n$-tuples of the form
  \[
  \begin{pmatrix}
    1&x_1\\ 2&x_2\\ \vdots\\ n&x_n\\
    1&x_{n+1}\\ 2&x_{n+2}\\ \vdots\\ n&x_{2n}\\
    1&x_{2n+1}\\ 2&x_{2n+2}\\ \vdots\\ n&x_{3n}\\
  \end{pmatrix}
  \]
  such that
  \begin{align}
    &x_{kn+1} \equiv_2x_{kn+2} \equiv_2 \dots \equiv_2 x_{kn+n},
      &\text{for $k=0,1,2$, and} \label{eqn:bits} \\
    &\sum_{i=1}^{3n} x_i = 2.\label{eqn:strange-sum}
  \end{align}
  The three equations from (\ref{eqn:bits}) mean that the least significant bits of the
  arithmetic parts of the first $n$ entries agree and similarly for the second
  and the last $n$ entries; equation (\ref{eqn:strange-sum}) can be viewed
  as a~combined parity check on all involved bits.

  \begin{claim}
    The relation $R$ is a~subuniverse of $(\m a_n)^{3n}$.
  \end{claim}
  \begin{proof}
    By the symmetry of the $t_i$'s and $R$, it is enough to show that $t_1$
    preserves $R$. Let us take three arbitrary members of $R$:
    \[
    \begin{pmatrix}
      1&x_{1,1}\\ 2&x_{1,2}\\ \vdots\\ n&x_{1,n}\\
      1&x_{1,n+1}\\ 2&x_{1,n+2}\\ \vdots\\ n&x_{1,2n}\\
      1&x_{1,2n+1}\\ 2&x_{1,2n+2}\\ \vdots\\ n&x_{1,3n}\\
    \end{pmatrix},
    \begin{pmatrix}
      1&x_{2,1}\\ 2&x_{2,2}\\ \vdots\\ n&x_{2,n}\\
      1&x_{2,n+1}\\ 2&x_{2,n+2}\\ \vdots\\ n&x_{2,2n}\\
      1&x_{2,2n+1}\\ 2&x_{2,2n+2}\\ \vdots\\ n&x_{2,3n}\\
    \end{pmatrix},
    \begin{pmatrix}
      1&x_{3,1}\\ 2&x_{3,2}\\ \vdots\\ n&x_{3,n}\\
      1&x_{3,n+1}\\ 2&x_{3,n+2}\\ \vdots\\ n&x_{3,2n}\\
      1&x_{3,2n+1}\\ 2&x_{3,2n+2}\\ \vdots\\ n&x_{3,3n}\\
    \end{pmatrix}
    \]
    and apply $t_1$ to them to get:
    \begin{equation}
      \vec r =
      \begin{pmatrix}
        1&x_{1,1}-x_{2,1}+x_{3,1}\\
        2&x_{1,2}\oplus x_{2,2}\oplus x_{3,2}\\
         &\vdots\\
        n&x_{1,n}\oplus x_{2,n}\oplus x_{3,n}\\
        1&x_{1,n+1}-x_{2,n+1}+x_{3,n+1}\\
        2&x_{1,n+2}\oplus x_{2,n+2}\oplus x_{3,n+2}\\
         &    \vdots\\
        n&x_{1,2n}\oplus x_{2,2n}\oplus x_{3,2n}\\
        1&x_{1,2n+1}-x_{2,2n+1}+x_{3,2n+1}   \\
        2&x_{1,2n+2}\oplus x_{2,2n+2}\oplus x_{3,2n+2}\\
         &   \vdots\\
        n&x_{1,3n}\oplus x_{2,3n}\oplus x_{3,3n}\\
      \end{pmatrix}
      \label{eqn:r}
    \end{equation}
    We want to verify that $\vec r\in R$. First note that (\ref{eqn:bits}) is
    satisfied: This follows from the fact that $x-y+z$ and $x\oplus y\oplus z$
    give the same result modulo 2, and the assumption that the original three
    tuples satisfied (\ref{eqn:bits}).

    What remains is to verify the property~(\ref{eqn:strange-sum}). If in the
    equality~(\ref{eqn:r}) above we replace the operations $\oplus$ by $-$ and
    $+$, verifying~(\ref{eqn:strange-sum}) is easy: The sum of the
    arithmetic parts of such a modified tuple is
    \begin{equation}
      \sum_{j=1}^{3n} (x_{1,j}-x_{2,j}+x_{3,j})=
      \sum_{j=1}^{3n}
      x_{1,j}-\sum_{j=1}^{3n}x_{2,j}+\sum_{j=1}^{3n}x_{3,j}=2-2+2=2.
      \label{eqn:2}
    \end{equation}
    This is why we need to examine the difference between the $\oplus$-based
    and $+$-based Maltsev operations. For $k\in \{0,1,2\}$ and $i\in
    \{1,\dots,n\}$ we let
    \[
      c_{k,i} = (x_{1,kn+i} \oplus x_{2,kn+i} \oplus x_{3,kn+i}) -
      (x_{1,kn+i} - x_{2,kn+i} + x_{3,kn+i})
    \]
   By the second part of Observation~\ref{obs:maltsev-diff}, $c_{k,i}$ does not
   depend on $i$ (changing $i$ does not change
    the $x_{j,kn+i}$'s modulo $\equiv_2$ by condition~(\ref{eqn:bits}) in the
    definition of $R$). Hence we can write just $c_k$ instead of $c_{k,i}$.

    Using $c_0$, $c_1$, and
    $c_2$ to adjust for the differences between the two Maltsev operations, we can express the
    sum of the arithmetic parts of the tuple $\vec{r}$ as
    \[
      \sum_{j=1}^{3n} (x_{1,j}-x_{2,j}+x_{3,j})+\sum_{i=2}^{n}
      c_0+\sum_{i=2}^{n} c_1+\sum_{i=2}^{n} c_2
      = 2+(n-1)(c_0+c_1+c_2)
    \]
    where we used~(\ref{eqn:2}) to get the right hand side. We chose $n$ odd,
    hence $n-1$ is even and each $c_k$ is even by
    Observation~\ref{obs:maltsev-diff}, so $(n-1)c_k=0$ for any $k=0,1,2$. We see that the sum of the
    arithmetic parts of $\vec{r}$ is equal to 2 which concludes the proof
    of~(\ref{eqn:strange-sum}) for the tuple~$\vec r$ and we are done.
%
  \end{proof}

  It is easy to see that
  \[
  \begin{pmatrix}
    1&0\\ 2&0\\ \vdots\\ n&0\\
    1&1\\ 2&1\\ \vdots\\ n&1\\
    1&1\\ 2&1\\ \vdots\\ n&1\\
  \end{pmatrix},
  \begin{pmatrix}
    1&1\\ 2&1\\ \vdots\\ n&1\\
    1&0\\ 2&0\\ \vdots\\ n&0\\
    1&1\\ 2&1\\ \vdots\\ n&1\\
  \end{pmatrix},
  \begin{pmatrix}
    1&1\\ 2&1\\ \vdots\\ n&1\\
    1&1\\ 2&1\\ \vdots\\ n&1\\
    1&0\\ 2&0\\ \vdots\\ n&0\\
  \end{pmatrix}\in R,
  \quad\text{and}\quad
  \begin{pmatrix}
    1&0\\ 2&0\\ \vdots\\ n&0\\
    1&0\\ 2&0\\ \vdots\\ n&0\\
    1&0\\ 2&0\\ \vdots\\ n&0\\
  \end{pmatrix}\notin R.
  \]
  However, the last tuple can be obtained from the first three by applying any
  minority operation on the set $A_n$ coordinate-wise. From this we conclude
  that $\m a_n$ does not have a~minority term.
\end{proof}

We note that the above construction of $\m a_n$ makes sense for $n$ even as well
and claim that these algebras also have the same key features, namely, by
construction, they have plenty of `local' minority term operations but they do
not have minority terms.  The verification of this last fact for $n$ even is
similar, but slightly more technical than for $n$ odd, and we omit the proof here.

The algebras $\m a_n$ can also be used to witness that having a~lot of local
minority-majority terms does not guarantee the presence of an actual
minority-majority term. By padding with dummy variables, any local minority
term of an algebra $\m a_n$ is also a~term that locally satisfies the
minority-majority term equations.  But since each $\m a_n$ has a~Maltsev term
but not a~minority term, then by Theorem~\ref{thm:join} it follows that $\m
a_n$ cannot have a~minority-majority term.

\section{Deciding minority in idempotent algebras is in \compNP}\label{sec:np}

The results from the previous section imply that one cannot base an efficient
test for the presence of a~minority term in a~finite idempotent algebra on
checking if it has enough local minority terms.   This does not rule out that
the problem is in the class \compP, but to date no other approach to showing
this has worked.  As an intermediate result, we show, at least, that this
decision problem is in \compNP{} and so cannot be \compEXPTIME-complete (unless
$\compNP=\compEXPTIME$).

We first show that an instance $\m a$ of the decision problem \minority\  can
be expressed as a~particular instance of the subpower membership problem for
${\m a}$.

\begin{definition}\label{defSMP}
Given a~finite algebra $\m a$, the \emph{subpower membership problem} for $\m
a$, denoted by $\smp(\m a)$, is the following decision problem:
\begin{itemize}
  \item INPUT: $\vec a_1, \dots, \vec a_k, \vec b \in A^n$
  \item QUESTION: Is $\vec b$ in the subalgebra of $\m a^n$ generated by
    $\{\vec a_1, \dots, \vec a_k\}$?
\end{itemize}
\end{definition}

To build an instance of $\smp(\m a)$ expressing that $\m a$ has a~minority
term, let $I =\{(a,b,c)\mid \mbox{$a, b, c \in A$ and $|\{a,b,c\}| \le 2$}\}$.
So $|I| = 3|A|^2 - 2|A|$. For $(a,b,c) \in I$, let $\min(a,b,c)$ be the minority
element of this triple.  So
\[
  \min(a,b,b) = \min(b,a,b) = \min(b,b,a) = \min(a,a,a) = a.
\]
For $1 \le i \le 3$, let $\vec \pi_i \in A^I$ be defined by $\vec \pi_i(a_1,
a_2, a_3) = a_i$ and define $\vec \mu_A \in A^I$ by $\vec \mu_A(a_1, a_2, a_3)
= \min(a_1, a_2, a_3)$, for all $(a_1, a_2, a_3) \in I$.  Denote the instance
$\vec \pi_1$, $\vec \pi_2$, $\vec \pi_3$, and $\vec \mu_A$ of $\smp(\m a)$ by
$\min(\m a)$.

\begin{proposition}\label{min-instance}
  An algebra $\m a$ has a~minority term if and only if  $\vec \mu_A$ is
  a~member of the subpower of $\m a^I$ generated by $\{\vec \pi_1, \vec \pi_2,
  \vec \pi_3\}$, i.e., if and only if $\min(\m a)$ is a~`yes' instance of
  $\smp(\m a)$ when $\m a$ is finite.
\end{proposition}

\begin{proof}
  If $m(x,y,z)$ is a~minority term for $\m a$, then applying $m$
  coordinatewise to the generators $\vec \pi_1$, $\vec \pi_2$, $\vec \pi_3$
  will produce the element $\vec \mu_A$.  Conversely, any term that produces
  $\vec \mu_A$ from these generators will be a~minority term for $\m a$.
\end{proof}

Examining the definition of $\min(\m a)$, we see that the parameters from
Definition~\ref{defSMP} are $k=3$ and $n=3|A|^2-2|A|$, which is (for algebras
with at least one at least unary basic operation) polynomial in $\|\m A\|$.
For $\m a$ idempotent, we can in fact improve $n$ to $3|A|^2-3|A|$,
since then we do not need to include in $I$ entries of the form $(a,a,a)$.

In general, it is known that for some finite algebras the subpower membership problem can be
\compEXPTIME-complete~\cite{Kozik2008} and that for some others, e.g., for any
algebra that has only trivial or constant basic operations,  it lies in the
class \compP.  In~\cite{Mayr2012}, P.\ Mayr shows that when $\m a$ has a~Maltsev
term, then $\smp(\m a)$ is in \compNP. We claim that a careful reading of Mayr's proof reveals that in fact the following uniform version of the subpower membership problem, where the algebra $\m a$ is considered as part of the input, is also in \compNP.

\begin{definition}
Define \smpun\ to be the following decision problem:
\begin{itemize}
  \item INPUT:  A~list of tables of basic operations of an algebra~$\m A$ that includes a~Maltsev operation, and $\vec a_1, \dots, \vec a_k, \vec b \in A^n$.
  \item QUESTION: Is $\vec b$ in the subalgebra of $\m a^n$ generated by
    $\{\vec a_1, \dots, \vec a_k\}$?
\end{itemize}
\end{definition}
We base the main result of this section
on the following.
\begin{theorem}[see \cite{Mayr2012}]\label{smpun}
  The decision problem \smpun\ is in the class \compNP.
\end{theorem}

While this theorem is not explicitly stated in \cite{Mayr2012}, it can be seen
that the runtime of the verifier that Mayr constructs for the problem $\smp(\m
a)$, when $\m a$ has a Maltsev term, has polynomial dependence on the size of
$\m a$ in addition to the size of the input to $\smp(\m a)$.  We stress that
Mayr's verifier requires that the table for a Maltsev term of $\m a$ is given as part of
the description of $\m a$.


\begin{theorem}\label{NP} The decision problem \minority\ is in the class \compNP.
\end{theorem}

\begin{proof}
To prove this theorem, we provide a polynomial reduction $f$ of \minority\ to \smpun.  By Theorem~\ref{smpun}, this will suffice.
Let $\m a$ be an instance of \minority, i.e., a~finite
  idempotent algebra that has at least one operation.

  We first check, using the polynomial-time
  algorithm from Corollary~\ref{maltsevterm}, to see if $\m a$ has a~Maltsev
  term.  If it does not, then $\m a$ will not have a~minority term, and in this case we
  set $f(\m a)$ to be some fixed `no' instance of \smpun.  Otherwise, we augment the list of basic operations of $\m a$ by
  adding the Maltsev operation on $A$ that the algorithm produced.  Denote
  the resulting (idempotent) algebra by $\m a'$ and note that $\m a'$ can be constructed from $\m a$ by a polynomial-time algorithm.
   Also, note that $\m a'$ is term equivalent to
  $\m a$ and so the subpower membership problem is the same for both
  algebras.

  If we set $f(\m a)$ to be the instance of \smpun\ that consists of  the~list of tables of basic operations of~$\m A'$ along with
  $\min(\m a)$ then we have, by Proposition~\ref{min-instance}, that $f(\m a)$ is a `yes' instance of \smpun\ if and only if $\m a$ has a minority term.  Since the construction of $f(\m a)$ can be carried out by a procedure whose runtime can be bounded by a polynomial in $\|\m a\|$, we have produced a polynomial reduction of  \minority\ to \smpun, as required.
 \end{proof}

\section{Conclusion}

While Theorem~\ref{NP} establishes that testing for a~minority term for finite
idempotent algebras is not as hard as it could be, the true complexity of this
decision problem is still open.  Our proof of this theorem closely ties the
complexity of {\minority} to the complexity of the subpower membership problem
for finite Maltsev algebras and specifically to the problem \smpun. Thus any progress on determining the complexity of
$\smp(\m a)$ for finite Maltsev algebras may have a~bearing on the complexity
of {\minority}.
%
There has certainly been progress on the algorithmic side of $\smp$;
a~major recent paper has shown in particular that $\smp(\m a)$ is tractable for
$\m a$ with cube term operations (of which a Maltsev term operation is
a~special case) as long as $\m a$ generates a residually small
variety~\cite{BMS18} (the statement from the paper is actually stronger than
this, allowing multiple algebras in place of $\m a$).

In Section~\ref{join} we introduced the notion of a~minority-majority term and
showed that if testing for such a~term for finite idempotent algebras could be
done by a~polynomial-time algorithm, then \minority\ would lie in the
complexity class \compP. This is why we conclude our paper with a~question
about deciding minority-majority terms.

\begin{open-problem*}
  What is the complexity of deciding if a~finite idempotent algebra has
  a~minority-majority term?
\end{open-problem*}

%

\bibliography{minority}

\end{document}